\documentclass[reqno]{amsart}

\usepackage{amssymb,amsmath,amsthm,amsfonts,fancyhdr}
\usepackage{setspace}
\usepackage{cite}

\newtheorem{thm}{Theorem}

\newtheorem{lemma}{Lemma}
\newtheorem{prop}{Proposition}

\newtheorem{ex}{Example}
\newtheorem{qu}{Question}

\begin{document}

\title{Infinity harmonic functions over exterior domains}

\author{Guanghao Hong $\cdot$
 Yizhen Zhao}

\address{School of Mathematics and Statistics, Xi'an Jiaotong University, Xi'an, P.R.China 710049.}

\email{ghhongmath@xjtu.edu.cn} \email{shadow19971202@163.com}

\begin{abstract}
In this paper, we study the infinity harmonic functions with linear growth rate at infinity defined on exterior domains. We show that such functions must be asymptotic to planes or cones at infinity. We also establish the solvability of Dirichlet problems for exterior domains.
\end{abstract}



\maketitle

\section{Introduction}

Let $\Omega \subset \mathbb{R}^n$ be an open set, an infinity harmonic function (IHF) $u\in C(\Omega)$ is a viscosity solution of the infinity Laplace equation
$$\triangle_{\infty}u:=\sum_{i,j}u_{x_i}u_{x_j}u_{x_ix_j}=0.$$
An extremely important characteristic property of infinity harmonic functions is the \textit{comparison with cones property (CCP)}: $\forall$ $V\subset\subset\Omega$ and $c(x)=a|x-x_0|+b$
\begin{eqnarray*}
c(x)\geq u(x) \ \mbox{on}\ \partial(V\backslash\{x_0\}) & \Rightarrow & c(x)\geq u(x) \ \mbox{in}\ V;\\
c(x)\leq u(x) \ \mbox{on}\ \partial(V\backslash\{x_0\}) & \Rightarrow & c(x)\leq u(x) \ \mbox{in}\ V.
\end{eqnarray*}
We refer the readers to [ACJ04][C08][W09][L16] for comprehensive expositions of this topic.

In this paper, we focus on the infinity harmonic functions over exterior domains. Let $A\subset \mathbb{R}^n$ be a bounded closed set. For simplicity, we assume $0\in A\subset B_1$. Let $\Omega:=A^c=\mathbb{R}^n\backslash A$ and $u\in C(\bar{\Omega)}$ be an IHF satisfying $\limsup\limits_{x\rightarrow\infty}\frac{|u(x)|}{|x|}<+\infty$. Denote $m^+:=\max\limits_{\partial\Omega}u$ and $m^-:=\min\limits_{\partial\Omega}u$. For $r>1$, we define
\begin{eqnarray*}
S^+_r:=\frac{\max(\max\limits_{\partial B_r}u,m^+)-m^+}{r} \ \ \mbox{and}\ \ S^-_r:=\frac{m^--\min(\min\limits_{\partial B_r}u,m^-)}{r}.
\end{eqnarray*}
By \textit{CCP}, both $S^+_r$ and $S^-_r$ are nondecreasing with respect to $r$. We define $S^{\pm}_{\infty}:=\lim\limits_{r\rightarrow\infty}S^{\pm}_r$ and $S_{\infty}:=\max(S^+_{\infty},S^-_{\infty})$. It is not difficult to see that $$S_{\infty}=\limsup_{x\rightarrow\infty}\frac{|u(x)|}{|x|}<+\infty$$ and
\begin{eqnarray*}
m^--S^-_{\infty}|x| \leq u(x)\leq m^++S^+_{\infty}|x| \ \ \mbox{in}\ \Omega.
\end{eqnarray*}

The main results of this paper are the following two theorems. The first one give a description of the asymptotic behavior at infinity of an exterior IHF with linear growth rate. The second one says that given the prescribed asymptotic behaviors there exist such IHFs.
\begin{thm}
Let $u$ be as above, then exact one of the four situations happens.

(i) $S_{\infty}=0$, $m^-\leq u(x)\leq m^+$ in $\Omega$;

(ii) $S^-_{\infty}<S^+_{\infty}$, $u(x)-S^+_{\infty}|x|$ attains its maximum and minimum on $\partial \Omega$;

(iii) $S^-_{\infty}>S^+_{\infty}$, $u(x)+S^-_{\infty}|x|$ attains its maximum and minimum on $\partial \Omega$;

(iv) $S^-_{\infty}=S^+_{\infty}>0$, there exists $a\in \mathbb{R}^n$ with $|a|=S_{\infty}$ such that
\begin{eqnarray*}
 u(x)=a\cdot x +o(|x|) \ \ \mbox{as}\ x\rightarrow\infty.
\end{eqnarray*}
\end{thm}

 The proofs of (ii) and (iii) of Theorem 1 rely on a key result in [SWY08]. The proof of (iv) is the real contribution of this paper. The proof is divided into two steps. The first step is to show the blow downs of $u$ are linear functions. This can be done in the same way of proving blow ups are linear. The method is standard now. The second step is to show the uniqueness of the blow downs. This was a challenging task for us. In [HZ18], we verified the uniqueness of the blow downs for the entire IHFs with linear growth rate by a similar argument from [ES11]. However, this argument cannot be carried to the case of exterior IHFs. The solution we finally found was that we can manage to place an entire IHF either below or above the exterior IHF. This implies the uniqueness of the blow downs for the exterior IHF $u$ clearly. We also used this idea in [HY18].

\begin{thm}
Given any $g\in C(\partial \Omega)$, we have the following.

(i) For any $\lambda\in \mathbb{R}$, there exists an IHF $u\in C(\bar{\Omega})$ satisfying $u|_{\partial \Omega}=g$ and $u(x)-\lambda|x|$ attains its maximum and minimum on $\partial \Omega$. In the case of $\lambda=0$ such $u$ is unique.

(ii) For any $a\in \mathbb{R}^n$ with $|a|>0$, there exists an IHF $u\in C(\bar{\Omega})$ satisfying $u|_{\partial \Omega}=g$ and
$u(x)-a\cdot x$ attains its maximum and minimum on $\partial \Omega$.
\end{thm}

The paper is organized as follows. In Section 2, we state and prove some preliminary results. In Sections 3 and 4, we prove Theorems 1 and 2 separately. In section 5, we analyze an interesting counterexample from [SWY08] in order to show that given $g$ and $a$ the IHFs satisfying $u|_{\partial \Omega}=g$ and  $u(x)=a\cdot x +o(|x|)$ as $x\rightarrow\infty$ are not unique in general. Until now it is not clear for us whether the solutions in (i) (in the case of $\lambda\neq 0$ ) and (ii) of Theorem 2 are unique.

\section{Preliminaries}
\begin{prop}
Let $u$ be given as in Theorem 1, then $\lim\limits_{r\rightarrow\infty}Lip(u,B^c_r)=S_{\infty}$.
\end{prop}
\begin{proof}
We first prove $S_{\infty}\leq\lim\limits_{r\rightarrow\infty}Lip(u,B^c_r)$. Suppose $\lim\limits_{r\rightarrow \infty}Lip(u,B^c_r)=L<+\infty$. Given $\epsilon>0$, there exists $R_{\epsilon}>0$ satisfying $Lip(u,B^c_r)
\leq L+\epsilon$ for $r\geq R_{\epsilon}$. Choose a point $x_0\in \partial B_{R_{\epsilon}}$, then
\[\frac{|u(x)|}{|x|}\leq\frac{|u(x)-u(x_0)|}{|x-x_0|}\frac{|x-x_0|}{|x|}+\frac{|u(x_0)|}{|x|}.\]
Letting $x\rightarrow \infty$, it follows that
\[S_{\infty}=\limsup\limits_{x\rightarrow \infty}\frac{|u(x)|}{|x|}\leq L+\epsilon.\]

Now we prove $\lim\limits_{r\rightarrow\infty}Lip(u,B^c_r)\leq S_{\infty}$. Given $\epsilon>0$, there exists $R_{\epsilon}>0$ satisfying $$\frac{S_{\infty}R+m^+-m^-}{R-1}\leq S_{\infty}+\epsilon\ \ \mbox{for}\ R\geq R_{\epsilon}.$$  We will show that $Lip(u,B^c_{R_{\epsilon}})\leq S_{\infty}+\epsilon$. Given any two point $y,z\in B^c_{R_{\epsilon}}$, there exists sufficiently large $R>\max(|y|,|z|)$ satisfying $$\frac{|u(x)-u(y)|}{|x-y|}\leq \frac{|u(x)|}{|x|}\frac{|x|}{|x-y|}+\frac{|u(y)|}{|x-y|} \leq S_{\infty}+\epsilon$$ for all $x\in \partial B_R$. On the other hand, $$\frac{|u(x)-u(y)|}{|x-y|}\leq \frac{S_{\infty}|y|+m^+-m^-}{|y|-1} \leq S_{\infty}+\epsilon$$ for all $x\in \partial B_1$. That is to say, $$|u(x)-u(y)|\leq (S_{\infty}+\epsilon)|x-y|$$ for all $x\in \partial (B_R\backslash B_1)$. By \textit{CCP}, we have $$|u(x)-u(y)|\leq (S_{\infty}+\epsilon)|x-y|$$ for all $x\in B_R\backslash B_1$. Especially, $$|u(z)-u(y)|\leq (S_{\infty}+\epsilon)|z-y|.$$
\end{proof}

The following theorem is the main technical result in [SWY08] (Theorem 1.3), which will be used in our proofs of (ii) and (iii) of Theorem 1.

\begin{thm}
Suppose $w\in C(\mathbb{R}^n)$ satisfies the following:

(\romannumeral 1) $Lip(w, \mathbb{R}^n)=1;$

(\romannumeral 2) for some $M\in \mathbb{R}$ and $\epsilon>0$, \[w(x)\leq M+(1-\epsilon)|x|\ \text{for all} \ x\in\mathbb{R}^n;\]

(\romannumeral 3) $w$ is an infinity harmonic function in $\mathbb{R}^n\backslash \{0\}.$

\noindent Then
\[w(x)=w(0)-|x|.\]
\end{thm}

The following theorem \footnote{After the online publication of [HZ18], we learned that this result has already appeared in [MWZ16].} is from [MWZ16] (Theorem 1.1) and [HZ18] (Theorem 2), which will be used in the proof of (iv) of Theorem 1.

\begin{thm}
Let $w$ be an IHF in $\mathbb{R}^n$ with $Lip(w, \mathbb{R}^n)<+\infty$. Then there exists $a\in \mathbb{R}^n$ with $|a|=Lip(w, \mathbb{R}^n)$ such that
\begin{eqnarray*}
 w(x)=a\cdot x +o(|x|) \ \ \mbox{as}\ x\rightarrow\infty.
\end{eqnarray*}
\end{thm}

The following theorem is from [CGW07] (Theorem 3.2), which will be used in the proof of the uniqueness part of (i) (in the case of $\lambda=0$) of Theorem 2.
\begin{thm}
Let $U$ be unbounded and $\partial U$ be bounded. Let $u,v\in C(\overline{U})$, and $\Delta_{\infty}u\geq 0,\ \Delta_{\infty}v\leq 0$ in $U$. Assume also that
\[\limsup\limits_{x\rightarrow\infty}\frac{u(x)}{|x|}\leq 0\ \text{and}\ \liminf\limits_{x\rightarrow\infty}\frac{v(x)}{|x|}\geq 0,
\]
Then
\[u(x)-v(x)\leq \max\limits_{\partial U}(u-v)\ \text{for}\ x\in U.\]
\end{thm}

The following lemma is well known and very frequently used in the study of IHFs (see, for example, [C08]).
\begin{lemma}
\label{l1}Suppose $w\in C(\mathbb{R}^n)$ satisfies:

 (\romannumeral 1) $\text{Lip}(w,\mathbb{R}^n)\leq 1$;

 (\romannumeral 2) There is a unit vector $e\in \mathbb{R}^n$ such that
 \[w(te)=t,\ \forall  t\in \mathbb{R}.\]
 Then $w(x)=e\cdot x$ for all $x\in \mathbb{R}^n$.
 \end{lemma}

\section{Proof of Theorem 1}
\begin{proof}[Proof of Theorem 1]
(i) is obvious. (ii) and (iii) are symmetric. So we only need to prove (ii) and (iv).

We first prove (ii). For simplicity, we assume $S^+_{\infty}=1$ and $S^-_{\infty}=\lambda\in [0,1)$. For any sequence $1<r_k\rightarrow +\infty$, define $v_k(x):=\frac{u(r_k x)}{r_k}$. By Proposition 1, on any compact set $K\subset \mathbb{R}^n\backslash\{0\}$, $v_k(x)$ are uniformly bounded and equi-continuous. Hence (up to a subsequence) $$v_k(x)\rightarrow V(x)\ \ \mbox{locally uniformly in}\ \mathbb{R}^n\backslash\{0\}.$$ It is easy to see that $V(0)=0$, $Lip(V, \mathbb{R}^n)\leq 1$ (by Proposition 1), $V(x)\geq -\lambda|x|$ and $V$ is an IHF in $\mathbb{R}^n\backslash\{0\}$. For each $k$, there is $e_k\in\partial B_1$ satisfying $u(r_k e_k)=\max\limits_{\partial B_{r_k}}u$ and hence $$\frac{u(r_k e_k)-m^+}{r_k}=S^+_{r_k}\rightarrow 1.$$ Up to a subsequence, $e_k\rightarrow e$. So $V(e)=1$ and hence $Lip(V, \mathbb{R}^n)=1$. From Theorem 3, $V(x)=|x|$.

Denote $\max\limits_{\partial\Omega}(u(x)-|x|)=c^+$ and $\min\limits_{\partial\Omega}(u(x)-|x|)=c^-$. Then $$c^- +|x|\leq u(x)\leq c^+ +|x| \ \ \mbox{on}\  \partial\Omega.$$ For any $\epsilon>0$, there is $\bar{k}$ such that $$c^-+(1-\epsilon)|x|\leq u(x)\leq c^++(1+\epsilon)|x| \ \  \mbox{on} \ \partial B_{r_k}$$ for all $k\geq \bar{k}$. By \textit{CCP}, $$c^-+(1-\epsilon)|x|\leq u(x)\leq c^++(1+\epsilon)|x|\ \ \mbox {in}\ \Omega.$$ Letting $\epsilon\rightarrow 0$, we have $$c^-+|x|\leq u(x)\leq c^++|x| \ \ \mbox {in}\ \Omega.$$

Now we prove (iv). For simplicity, we assume $S^-_{\infty}=S^+_{\infty}=1$. For any sequence $1<r_k\rightarrow +\infty$, define $v_k(x):=\frac{u(r_k x)}{r_k}$. We still have (up to a subsequence) $$v_k(x)\rightarrow V(x)\ \ \mbox{locally uniformly in}\ \mathbb{R}^n\backslash\{0\}.$$ It can also be verified that $V(0)=0$, $Lip(V, \mathbb{R}^n)\leq 1$ and $V$ is an IHF in $\mathbb{R}^n\backslash\{0\}$. Fix a $R>1$. For each $k$, there are $e_k^+, e_k^- \in\partial B_1$ such that $u(r_k R e_k^+)=\max\limits_{\partial B_{r_k R}}u$ and $u(r_k R e_k^-)=\min\limits_{\partial B_{r_k R}}u$. Hence $$\frac{u(r_k R e_k^+)-m^+}{r_k R}=S^+_{r_k R}\rightarrow 1$$ and $$\frac{u(r_k R e_k^-)-m^-}{r_k R}=-S^-_{r_k R}\rightarrow -1.$$ Up to a subsequence, $e_k^+\rightarrow e^+_{[R]}$ and $e_k^-\rightarrow e^-_{[R]}$. So $V(R e^+_{[R]})=R$ and $V(R e^-_{[R]})=-R$. Since $Lip(V, \mathbb{R}^n)\leq 1$, we have $$2R=V(R e^+_{[R]})-V(R e^-_{[R]})\leq |R e^+_{[R]}-R e^-_{[R]}|\leq 2R.$$ This implies $-e^-_{[R]}=e^+_{[R]}:=e_{[R]}$ and we have $V(te_{[R]})=t$ for $t\in [-R,R]$. In fact, we have this for any $R>1$. Choose another $\tilde{R}>1$, then $$\tilde{R}+R=V(\tilde{R} e_{[\tilde{R}]})-V(-R e_{[R]})\leq |\tilde{R} e_{[\tilde{R}]}-(-R e_{[R]})|\leq \tilde{R}+R.$$ This implies $e_{[\tilde{R}]}=e_{[R]}$. That is the vector $e_{[R]}$ is independent of $R$. We denote this vector as $e$ and we have $V(te)=t$ for $t\in (-\infty, +\infty)$.
By Lemma 1, $V(x)=e\cdot x$.

We have showed that the blow downs of $u$ are linear functions with slope 1. In order to get the conclusion of (iv) we have to show that the blow downs are unique.

For $k=2,3,\cdots$, let $w_k$ be the IHFs in $B_k$ satisfying $w_k=u$ on $\partial B_k$. For each $k$, either $\max\limits_{\partial \Omega}(w_k-u)\geq 0$ or $\min\limits_{\partial \Omega}(w_k-u)\leq 0$. So either $\max\limits_{\partial \Omega}(w_k-u)\geq 0$ or $\min\limits_{\partial \Omega}(w_k-u)\leq 0$ happens for infinitely many $k$. Let's assume the first case (the second case can also give the final conclusion in a similar way) and denote these $k$ as $k_j$. Define $$\tilde{w}_{k_j}:=w_{k_j}-\max\limits_{\partial \Omega}(w_{k_j}-u),$$ then $\tilde{w}_{k_j}\leq u$ on $\partial (B_{k_j}\cap \Omega)$ (implying $\tilde{w}_{k_j}\leq u$ in $B_{k_j}\cap \Omega$ by comparison principle) and $\tilde{w}_{k_j}(y_{k_j})= u(y_{k_j})$ for some point $y_{k_j}\in\partial \Omega$.
Note that $$Lip(\tilde{w}_{k_j}, B_{k_j})=Lip(\tilde{w}_{k_j}, \partial B_{k_j})=Lip(u, \partial B_{k_j})\leq Lip(u, B^c_{k_j})\rightarrow 1.$$ So we have $|\tilde{w}_{k_j}(0)|\leq \max(|m^+|,|m^-|)+2$ for all large $k_j$. By Ascoli-Arzela's theorem, we have (up to a subsequence) $$\tilde{w}_{k_j}\rightarrow W\ \ \mbox{in}\ \mathbb{R}^n\ \ \mbox{locally uniformly}.$$
Here $W$ is an IHF in $\mathbb{R}^n$ satisfying $Lip(W, \mathbb{R}^n)\leq 1$ and $W\leq u$ in $\Omega$. By Theorem 4,
\begin{eqnarray*}
 W(x)=a\cdot x +o(|x|) \ \ \mbox{as}\ x\rightarrow\infty
\end{eqnarray*}
for some $a\in \mathbb{R}^n$ with $|a|=Lip(W, \mathbb{R}^n)$.

The fact $u\geq W$ in $\Omega$ implies that any blow down of $u$ $$V(x)=e\cdot x\geq a\cdot x.$$ So $e=a$ (implying $|a|=1$) and $V(x)=a\cdot x$.
\end{proof}

\section{Proof of Theorem 2}

\begin{proof}[Proof of Theorem 2]
We first prove (i). Let $g\in C(\partial \Omega)$ and $\lambda\in \mathbb{R}$ are given. Denote $\max\limits_{\partial\Omega}(g(x)-\lambda|x|)=c^+$ and $\min\limits_{\partial\Omega}(g(x)-\lambda|x|)=c^-$. For $k=2,3,\cdots$, let $u_k$ be the IHF in $B_k\cap \Omega$ satisfying $u_k=c^++\lambda|x|$ on $\partial B_k$ and $u_k=g$ on $\partial \Omega$. By \textit{CCP}, one can verify that on any compact set $K\subset\subset \Omega$, $Lip(u_k, K)\leq C(g,\lambda, K)$ and $\|u_k\|_{L^{\infty}(K)}\leq C(g,\lambda, K)$ for all large $k$. Therefore by Ascoli-Arzela's theorem, up to a subsequence, we have
\begin{eqnarray*}
 u_k\rightarrow u \ \ \mbox{locally uniformly in}\ \Omega.
\end{eqnarray*}
The function $u$ is an IHF in $\Omega$ satisfying $u=g$ on $\partial \Omega$. By comparison principle, we know that
\begin{eqnarray*}
c^-+\lambda|x|\leq u(x)\leq c^++\lambda|x| \ \ \mbox {in}\ \Omega.
\end{eqnarray*}

In case of $\lambda=0$, the uniqueness of solution $u$ follows from Theorem 5 directly.

The proof of (ii) is same with (i). We just need to replace $\lambda|x|$ with $a\cdot x$ in every steps of the above process.

\end{proof}

\section{Counterexamples}

In [SWY08], the authors constructed an IHF $U(x)$ in $\mathbb{R}^n\backslash\{0\}$. The function $U$ satisfies the following properties: $Lip(U, \mathbb{R}^n)=1$, $U(te_n)=t$ for $t\in (-\infty,0]$ and $U(e_n)=0$. Hence $U$ is neither linear nor a cone. We refer the readers to the original paper [SWY08] (Page 4) for the construction of $U$. Using the established results in this paper, we can get the following new fact on $U$.

\begin{prop}
$U(x)=e_n\cdot x+o(|x|)$ as $x\rightarrow \infty$.
\end{prop}

\begin{proof}
By Theorem 1 (in this case $\Omega=\mathbb{R}^n\backslash\{0\}$), $U$ is asymptotic to a plane or a cone at infinity. If $U$ is asymptotic to a cone, this cone can only be $-|x|$ since $U(te_n)=t$ for $t\leq 0$. By (iii) of Theorem 1, $U(x)=-|x|$. This is impossible since $U(e_n)=0$. So $U$ is asymptotic to a plane and this plane has slope less than or equal to 1 since $Lip(U, \mathbb{R}^n)=1$. Thus this plane can only be $e_n\cdot x$ since $U(te_n)=t$ for $t\leq 0$.
\end{proof}

We can use this function $U$ to show that given $g$ and $a$ the IHFs satisfying $u|_{\partial \Omega}=g$ and $u(x)=a\cdot x +o(|x|)$ as $x\rightarrow\infty$ are not unique in general.

\begin{ex}
 For the exterior domain $\Omega=\mathbb{R}^n\backslash\{0\}$, given $g=0$ on $\partial \Omega=\{0\}$ and $a=e_n$, both the two functions $U(x)$ and $V(x)=V(x',x_n):=-U(x',-x_n)$ satisfy $u|_{\partial \Omega}=g$ and $u(x)=a\cdot x +o(|x|)$ as $x\rightarrow\infty$.
\end{ex}

\begin{ex}
 For the exterior domain $\Omega=\mathbb{R}^n\backslash\{0, e_n\}$, given $g=0$ on $\partial \Omega=\{0, e_n\}$ and $a=e_n$, both the two functions $U(x)$ and $\tilde{V}(x)=V(x-e_n)=-U(x',1-x_n)$ satisfy $u|_{\partial \Omega}=g$ and $u(x)=a\cdot x +o(|x|)$ as $x\rightarrow\infty$.
\end{ex}

We can also use this function $U$ to illustrate a problem. From the construction (see [SWY08] Page 4), we know $U(x)\leq e_n\cdot x$, but $U(x)\geq e_n\cdot x-1$ does not hold. That is, considering $U$ as an exterior IHF in $\Omega=\mathbb{R}^n\backslash\{0, e_n\}$, $U(x)-e_n\cdot x$ does not attain its minimum on $\partial \Omega$. This illustrates that the conclusion ``$u(x)=a\cdot x +o(|x|)$ as $x\rightarrow\infty$'' in (iv) of Theorem 1 can not be improved to that ``$u(x)-a\cdot x$ attains its maximum and minimum on $\partial \Omega$''.

Finally, the following questions are interesting for us but so far we do not know the answers.

\begin{qu}
Is $U-e_n\cdot x$ bounded below? If $U-e_n\cdot x$ is not bounded, how big can this asymptotic error term $o(|x|)$ be?
\end{qu}

\begin{qu}
Are the solutions in (i) (in the case of $\lambda\neq 0$ ) and (ii) of Theorem 2 unique?
\end{qu}

\section*{Acknowledgements}
This work is supported by the National Nature Science Foundation of China: NSFC 11301411 and 11671316. The first author would like to thank Professor Dongsheng Li for many helpful conversations and encouragement. Part of this paper was completed during the first author's visit to University of Washington (Seattle). His visit was funded by China Scholarship Council. He would also like to thank Professor Yu Yuan for the invitation and to the Department of Mathematics for warm hospitality.

\bibliographystyle{elsarticle-num}

\begin{thebibliography}{99}

\bibitem [ACJ04]{ACJ04} Aronsson, Gunnar and Crandall, Michael G, and Juutinen, Petri: \textit{A tour of the theory of absolutely minimizing
functions}, Bull. Amer. Math. Soc. (N.S.) 41 (2004), no. 4, 439¨C505.

\bibitem [C08]{C08} Crandall, Michael G: \textit{A visit with the ¡Þ-Laplace equation}, Calculus of variations and nonlinear partial
differential equations, Lecture Notes in Math., vol. 1927, Springer, Berlin, 2008, pp. 75-122.

\bibitem [CGW07]{CGW07} Crandall, Michael G and Gunnarsson, Gunnar and Wang, Peiyong: \textit{Uniqueness of $\infty$-harmonic functions and the Eikonal equation}, Communications in Partial Differential Equations, \textbf{32} (2007), No.10, 1587-1615.

\bibitem [ES11]{ES11} Evans, L. C. and Smart, C. K.: \textit{Everywhere differentiability of infinity harmonic
functions}, Calc. Var. Partial Differ. Equ., \textbf{42} (2011), No.1-2, 289-299.

\bibitem [HY18]{HY18}Hong, Guanghao and Yuan, Yu: \textit{Maximal hypersurfaces over exterior domains}, preprint, 2018.

\bibitem [HZ18]{HZ18} Hong, Guanghao and Zhao, Yizhen: \textit{A Liouville theorem for infinity harmonic functions}, Manuscripta Mathematica, online, https://doi.org/10.1007/s00229-018-1099-8

\bibitem [L16]{L16} Lindqvist, Peter. \textit{Notes on the infinity Laplace equation}. Springer, 2016.

\bibitem [MWZ16]{MWZ16} Miao, Qianyun and Wang, Changyou and Zhou, Yuan:  \textit{$\infty$-harmonic functions with linear growth are differentiable
at $\infty$}, unpublished notes, 2016.

\bibitem [SWY08]{SWY08} Savin, Ovidiu and Wang, Changyou and Yu, Yifeng: \textit{Asymptotic Behavior of Infinity Harmonic Functions Near an Isolated Singularity}, International Mathematics Research Notices, (2008) , No.1, 121-136.

\bibitem [W09]{W09} Wang, Changyou.: \textit{An Introduction of Infinity Harmonic Functions}, unpublished lecture notes (2009).
\end{thebibliography}

\end{document}